\documentclass[10pt]{amsart}
\usepackage[english]{babel}
\usepackage{amssymb}
\usepackage{enumitem}
\usepackage{comment}
\usepackage{hyperref}
\usepackage[initials]{amsrefs}
\usepackage{tikz}
\usepackage{tikz-cd}
\usetikzlibrary{decorations.pathmorphing}
\usetikzlibrary{babel}
\usepackage{color}

\newcommand{\bbA}{{\mathbb A}}

\newcommand{\bbQ}{{\mathbb Q}}
\newcommand{\bbR}{{\mathbb R}}

\newcommand{\bbZ}{{\mathbb Z}}

\newcommand{\bfH}{{\mathbf H}}

\newcommand{\calV}{\mathcal{V}}

\newcommand{\calO}{\mathcal{O}}

\newcommand{\bs}{\backslash}

\newcommand{\id}{\operatorname{id}}
\newcommand{\im}{\operatorname{im}}

\newcommand{\pr}{\operatorname{pr}}

\newcommand{\Ad}{\operatorname{Ad}}
\newcommand{\Ker}{\operatorname{Ker}}

\newcommand{\SL}{\operatorname{\mathbf{SL}}}

\newcommand{\Aut}{\operatorname{Aut}}

\newcommand{\PGL}{\operatorname{\mathbf{PGL}}}
\newcommand{\GL}{\operatorname{\mathbf{GL}}}

\newcommand{\fin}{{\mathrm{fin}}}

\newcommand{\Commen}{\operatorname{Commen}}

\newcommand{\normal}{\triangleleft}

\newtheorem{theorem}{Theorem}[section]
\newtheorem{lemma}[theorem]{Lemma}

\newtheorem{corollary}[theorem]{Corollary}

\newtheorem{prop}[theorem]{Proposition}

\theoremstyle{definition}
\newtheorem{definition}[theorem]{Definition}
\newtheorem{defn}[theorem]{Definition}

\newtheorem{example}[theorem]{Example}

\newtheorem{remark}[theorem]{Remark}

\numberwithin{equation}{section}

\begin{document}

\title{An adelic arithmeticity theorem for lattices in products}

\author{Uri Bader}
\address{Weizmann Institute, Rehovot}
\email{uri.bader@gmail.com}

\author{Alex Furman}
\address{University of Illinois at Chicago, Chicago}
\email{furman@uic.edu}

\author{Roman Sauer}
\address{Karlsruhe Institute of Technology}
\email{roman.sauer@kit.edu}

\thanks{U.B. and A.F. were supported in part by the BSF grant 2008267.}
%\thanks{U.B was supported in part by the ISF grant 704/08.}
\thanks{A.F. was supported in part by the NSF grants DMS 1611765.}
%\thanks{R.S. gratefully acknowledges support from the \emph{Deutsche Forschungsgemeinschaft},
%made through grant SA 1661/1-2, during the initial phase of this project.}

%\subjclass[2000]{Primary 20F67; Secondary 55N99}
%\keywords{Hyperbolic groups, measure equivalence, simplicial volume}

\maketitle

\begin{abstract}
In this paper we prove that, under mild assumptions, a lattice in a product of semi-simple Lie group and a totally disconnected locally compact group is, in a certain sense, arithmetic.
We do not assume the lattice to be finitely generated or the ambient group to be compactly generated. 
\end{abstract}

%\tableofcontents

\section{Introduction and statement of the main result} % (fold)
\label{sec:introduction_and_statement_of_the_main_result}

\subsection{Motivation}\label{sub: intro-motivation}
Our goal in this paper is to show that, under mild conditions, a lattice~$\Gamma$ in a product of a semi-simple Lie group $H$ and a totally disconnected locally compact group $D$ is -- what we will call -- \emph{arithmetically constructed}. Under the assumption that $\Gamma<H\times D$ is irreducible and $\Gamma$ is finitely generated or $D$ is compactly generated such arithmeticity results  are obtained by Caprace-Monod in~\cite{Caprace+Monod:I}. In this work we drop these assumptions and obtain an adelic arithmeticity result for lattices in $H\times D$. An application will appear in~\cites{BFS:envelopes}. 

The exact formulation of the main result is given in Theorem~\ref{T:arithmetic-core} below.
Before formally stating it we will provide some motivation and an idea 
of what \emph{arithmetically constructed} means. 
A prototypical example is the lattice embedding
\[
	\PGL_2(\bbQ)<\PGL_2(\mathbb{A})\cong \PGL_2(\bbR)\times{\prod_{p}}'\PGL_2(\bbQ_p).
\]
Here the factor $\PGL_2(\bbR)$ is a connected, center-free semi-simple, real Lie group. Further, $\bbA$ denotes the ring of adeles. By $\prod'$ we denote restricted products in the sense of adeles: for almost all places $p$ the corresponding element lies in $\PGL_2(\bbZ_p)$.
A similar example is 
\[
\SL_2(\bbQ)<\SL_2(\mathbb{A})\cong \SL_2(\bbR)\times{\prod_{p}}'\SL_2(\bbQ_p).
\]
Despite the similarity, note that the image of $\SL_2(\bbQ)$ in $\PGL_2(\mathbb{A})$ under the obvious map
\[
	\SL_2(\mathbb{Q})\to \SL_2(\mathbb{A}) \to \PGL_2(\mathbb{A})
\]
is not a lattice. Indeed, it is of infinite index in the lattice $\PGL_2(\bbQ)$, as can be seen by the exact sequence
\[
 	1 \to \{+1,-1\} \to  \SL_2(\mathbb{Q}) \to \PGL_2(\bbQ) \to \bbQ^*/(\bbQ^*)^2 \to 1. 
\]
Intermediate subgroups between $\PGL_2(\bbQ)$ and the image of $\SL_2(\bbQ)$  are arithmetically constructed lattices in the 
sense of Definition~\ref{def:arithcon}. A special case of our main result identifies, under mild conditions, an arbitrary lattice in $\PGL_2(\bbR)\times D$ with an arithmetically constructed lattice. 

\subsection{Notation and basic setup}\label{sub: intro-notation}
Throughout, let $K$ be a number field with ring of integers $\calO$.
Let $\calV$ be the set of inequivalent valuations (places) of $K$, 
let $\calV^\infty$ denote the archimedean ones, and $\calV^\fin=\calV-\calV^\infty$ 
the non-archimedean ones (finite places). 
Given a subset $S\subset \calV$ we set
$S^\infty=S\cap \calV^\infty$ and $S^\fin=S\cap \calV^\fin$. 
We let 
\[
	\calO[S]:=\{x\in K \mid |x|_\nu\le 1\ \textrm{for\ all}\ \nu\in\calV^\fin-S\}
\] 
be the ring of $S$-integers in $K$.
For a semi-simple $K$-algebraic group $\mathbf{H}$ we denote 
by $\mathbf{H}_{S}$ the (truncated) $S$-adeles of $\mathbf{H}$, that is, the restricted product 
\begin{equation*}
	\mathbf{H}_{S}:=\bigl\{ (x_\nu)\in \prod_{\nu\in S}\mathbf{H}(K_\nu)\mid 
	x_\nu\in \mathbf{H}(\calO_\nu) \text{ for almost all $\nu\in S$}\bigr\}. 
\end{equation*}
We let  $\mathbf{H}_S^\#$ be the closure in $\mathbf{H}_{S}$ of the image of $\mathbf{H}(K)$ under the diagonal embedding.

For a field extension $k$ of $K$ we denote by $\mathbf{H}(k)^+<\mathbf{H}(k)$ the normal 
subgroup as defined in~\cite{borel+tits-abstract}*{Section~6}. 
Since $k$ is perfect, $\mathbf{H}(k)^+$ is the subgroup generated by all 
unipotent elements and it has finite index in $\mathbf{H}(k)$. 
If $\nu\in \calV^\infty$, then 
$\mathbf{H}(K_\nu)^+$ is just the connected component of the identity in the real Lie group $\mathbf{H}(K_\nu)$. 
We denote by $\mathbf{H}_{S}^+$ the restricted product 
\begin{equation*}
	\mathbf{H}_{S}^+:=\bigl\{ (x_\nu)\in \prod_{\nu\in S}\mathbf{H}(K_\nu)^+\mid 
	x_\nu\in \mathbf{H}(\calO_\nu) \text{ for almost all $\nu\in S$}\bigr\}. 
\end{equation*}
Note that if $S$ is finite, then 
\begin{equation*}
	\mathbf{H}_{S}=\prod_{\nu\in S}\mathbf{H}(K_\nu), \qquad
	\mathbf{H}_{S}^+=\prod_{\nu\in S}\mathbf{H}(K_\nu)^+.	
\end{equation*}
In this case
$\mathbf{H}_S^+<\mathbf{H}_S$ has 
finite index. In general, this is false for infinite~$S$.

\begin{defn} \label{def:compatible}
In the above context $S$ is said to be \emph{compatible} with $\mathbf{H}$ if the following hold: 
\begin{enumerate}
\item For every $\nu\in S$, $\mathbf{H}$ is $K_\nu$-isotropic.
\item $S$ contains all $\nu\in\calV^\infty$ for which $\mathbf{H}$ is $K_\nu$-isotropic. 
\item $S$ contains at least one finite and one infinite place. 
\end{enumerate}
\end{defn}

\subsection{Arithmetically constructed lattices}\label{sub: intro-arithmeticlattices}

The definition of arithmetically constructed lattices builds on the following proposition which we prove in Section~\ref{sec: arithmetic lattices}. 

\begin{prop}\label{ex:lattices}
For every $i\in\{1,\dots,n\}$ let $K_i$ be a number field, 
$\mathbf{H}_i$ be a connected, non-commutative, absolutely simple, adjoint $K_i$-group, and $S_i\subset \calV_i$ be a (possibly infinite) subset of places 
of $K_i$ which is compatible with $\bfH_i$. Then the following hold. 

\begin{enumerate}
\item 	$(\mathbf{H}_i)_{S_i^\fin}^+$ is an open subgroup of $(\mathbf{H}_i)_{S_i^\fin}^\#$ for every $i\in\{1,\dots,n\}$. 
\item For every subgroup 
\[A<\prod_{i=1}^n \left((\mathbf{H}_i)_{S_i^\fin}^\#/ (\mathbf{H}_i)_{S_i^\fin}^+\right)\]
let $D_A<\prod_{i=1}^n (\mathbf{H}_i)_{S_i^\fin}^\#$ be the preimage of $A$ under the quotient map, and let $\Gamma_A< \prod_{i=1}^n  \mathbf{H}_i(\calO_i[S_i])$ be the preimage of $D_A$ under 
the product of diagonal embeddings $\mathbf{H}_i(\calO_i[S_i])\to(\mathbf{H}_i)_{S_i^\fin}^\#$. 
The via the diagonal embedding 
\[ 
	\Gamma_A < \prod_{i=1}^n (\mathbf{H}_i)_{S_i^\infty}\times D_A 
\]
is a lattice.
\end{enumerate}
\end{prop}

\begin{defn} \label{def:arithcon}
A subgroup of $\prod_{i=1}^n \mathbf{H}_i(K_i)$ is an \emph{arithmetically constructed lattice} if it contains the group $\Gamma_A$ from Proposition~\ref{ex:lattices} as a subgroup of finite index for some subgroup 
\[A<\prod_{i=1}^n \left((\mathbf{H}_i)_{S_i^\fin}^\#/ (\mathbf{H}_i)_{S_i^\fin}^+\right).\]
\end{defn}

\begin{example}
	We continue the example from Subsection~\ref{sub: intro-motivation}. Here $K=K_1=\bbQ$, $n=1$, $\mathbf{H}=\mathbf{H}_1=\PGL_2$ and $S^\fin=S_1^\fin$ consists of all finite primes. Let $\bbA_\fin=\prod'_{p\ne\infty} \bbQ_p$ 
 be the ring of finite rational adeles. 
		The determinant 
	morphism $\det\colon \mathbf{G}_m(\bbA_\fin)\to \bbA_\fin^\ast/(\bbA_\fin^\ast)^2$ modulo squares descends to a morphism from 
	$\PGL_2(\bbA_\fin)$. The group $\bbA^\ast_\fin$ of ideles is the subset of $(x_p)\in \prod_{p\ne \infty} \bbQ_p^\ast$ such that $x_p\in \bbZ_p^\ast$ for almost all primes~$p$. Identifying the $p$-adic units $\bbQ_p^\ast$ with $\bbZ_p^\ast\times\bbZ$ we obtain an isomorphism 
	\[ \mathbf{G}_m(\bbA_\fin)\cong \prod_{p\ne \infty}\bbZ_p^\ast \times \bigoplus_{p\ne\infty} \bbZ\]
	of topological groups and a commutative diagram of topological 
	groups: 
	\[
	\begin{tikzcd}
          \GL_2(\bbA_\fin)\arrow[r, equal]\arrow[d] &\prod'\GL_2(\bbQ_p)\arrow[r, "\det"] & \prod_{p\ne \infty}\bbZ_p^\ast \times \bigoplus_{p\ne\infty} \bbZ\arrow[d, twoheadrightarrow] \\
          \PGL_2(\bbA_\fin)\arrow[r, equal] &\prod'\PGL_2(\bbQ_p) \arrow[r, "\det"] &\prod_{p\ne \infty}(\bbZ_p^\ast)/(\bbZ_p^\ast)^2 \times \bigoplus_{p\ne\infty} \bbZ/2\bbZ
    \end{tikzcd}
    \]
    In this example we have 
    \begin{align*} 
           \mathbf{H}_{S^\fin}^\# &= \operatorname{det}^{-1}\Bigl(\{1\}\times \bigoplus_{p\ne \infty} \bbZ/2\bbZ\Bigr)\\
           \mathbf{H}_{S^\fin}^+ &=\ker\det.
    \end{align*}
	Thus the quotient $\mathbf H_{S^\fin}^\#/\mathbf H_{S^\fin}^+$ appearing in the previous definition is here isomorphic to 
	$\bigoplus_{p\ne \infty} \bbZ/2\bbZ$. 
\end{example}

\subsection{Main result}
The main goal of this paper is to establish the following partial converse of Proposition~\ref{ex:lattices} which is an arithmeticity theorem for lattices in products. 
We use the abbreviation \emph{tdlc} for totally disconnected locally compact groups.

\begin{theorem}\label{T:arithmetic-core}\hfill{}\\
		Let $H$ be a connected, center-free, semi-simple, real Lie group 
		without compact factors,
		$D$ be a tdlc group and $\Gamma<H\times D$ be a lattice. Assume that 
		\begin{enumerate}
			\item the projection $\Gamma\to D$ has a dense image,
			\item the projection $\Gamma\to H$ has a dense image,
			\item the projection $\Gamma\to H$ is injective.
		\end{enumerate}
		Then there exist number fields $K_1,\dots,K_n$,  
		connected, non-commutative, absolutely simple, adjoint $K_i$-groups $\mathbf{H}_i$, 
		and (possibly infinite) compatible 
		subsets $S_i\subset \calV_i$ of places of $K_i$ (see Definition~\ref{def:compatible}) 
		with the following properties: 
		There is a topological isomorphism
		\begin{equation*}
			H \cong \prod_{i=1}^n (\mathbf{H}_i)_{S_i^\infty}^0
		\end{equation*}
		and a short exact sequence $1\to C\to D\to Q\to 1$ of
		topological groups, where $C$ is compact, and $Q$ is a closed intermediate subgroup  
		\[
			\prod_{i=1}^n (\mathbf{H}_i)_{S_i^\fin}^+<Q<\prod_{i=1}^n (\mathbf{H}_i)_{S_i^\fin}
		\]
		so that the inclusion $\Gamma< H \times D$ 
		is contained in the image of 
		\[
			\prod_{i=1}^n \mathbf{H}_i(K_i)\to \prod_{i=1}^n (\mathbf{H}_i)_{S_i},
		\]
		and, identifying $\Gamma$ with a subgroup of $\prod_{i=1}^n \mathbf{H}_i(K_i)$, 
		$\Gamma$ is an arithmetically constructed lattice as in Definition~\ref{def:arithcon}.
		
		Furthermore, compact generation of $D$, finite generation of $\Gamma$,
		and finiteness of all $S_i\subset \calV_i$, are equivalent conditions.
		If they hold, each $(\mathbf{H}_i)_{S_i}^+$ has finite index in $(\mathbf{H}_i)_{S_i}$, 
		so up to commensurability $\Gamma$ is the product of $S_i$-arithmetic lattices
		$\prod_{i=1}^n \mathbf{H}_i(\calO_i[S_i])$
		diagonally embedded in 
		\[
			H\times D/C\simeq \prod_{i=1}^n\mathbf{H}_{S_i}\simeq \prod_{i=1}^n\mathbf{H}_{S_i}^+
		\]
		where $\simeq$ means being commensurable. 
\end{theorem}

\begin{remark}
This result is an important ingredient in our work on lattice envelopes \cite{BFS:envelopes} 
(see also the announcement \cite{BFS:CRAS}).
In this application finite generation and irreducibility of $\Gamma$ are not known a priori.

For finitely generated $\Gamma$ or compactly generated $D$, such  results are proved by Caprace and Monod~\cite[Section~5]{Caprace+Monod:I}. 
Our approach is different except for step (i) below 
that uses Margulis' commensurator criterion for arithmeticity.

An adelic analog of Margulis' arithmeticity where ones assumes the ambient group of algebraic origin is proved by Oh~\cite{oh-adelic}. 
\end{remark}

\subsection{Strategy of proof}
The proof of Theorem~\ref{T:arithmetic-core} can be roughly divided into the following steps: 

\begin{enumerate}
	\item
	Choosing a compact open subgroup $U<D$, we observe that the projection $\Delta$ of $\Gamma\cap (H\times U)$ 
	to $H$ is a lattice in $H$, commensurated by the projection $\Lambda$ of all of $\Gamma$ to $H$.
	Then Margulis' commensurator superrigidity theorem provides the arithmetic structure of $\Delta$
	yielding the algebraic groups $\mathbf{H}_1,\dots,\mathbf{H}_n$, 
	the product of which defines the arithmetic structure on $H$.
	(The possible multitude of these groups corresponds to the possibility that $\Delta$ is reducible).
	\item
	The commensurator of $\Delta$ in $H$  
	can be identified with the image of $\prod\mathbf{H}_i(K_i)$,
	and $\Lambda\cong\Gamma$ is viewed as a subgroup of the latter.
	Considering the projections of $\Gamma$ to various finite places we identify 
	(see Theorem~\ref{thm: margulis lemma}) the sets $S_i^\fin$ of the non-archimedean places,
	so that on one hand 
	\[
		\Gamma<\prod_{i=1}^n \mathbf{H}_i(\calO_i[S_i]),
	\] 
	and on the other the closure $\overline{\Gamma}$ of the projection of $\Gamma$ 
	to $\prod_{i=1}^n (\mathbf{H}_i)_{S_i}$ contains $\prod_{i=1}^n (\mathbf{H}_i)_{S_i}^+$.
	(If $S_i$ is finite then $(\mathbf{H}_i)^+_{S_i}$ has finite index in $(\mathbf{H}_i)_{S_i}$).
	\item
	Next we construct a continuous surjective homomorphism with compact kernel
	from $D$ to the closure of the diagonal image of $\Gamma$ in $\prod (\mathbf{H}_i)_{S_i^\fin}$. 
	This step is conceptualized in the notion of \emph{representations of commensurable pairs} 
	(cf.~Section~\ref{sub:reps-of-comm}).
	This identifies $\Gamma$ as an arithmetically constructed lattice.
	\item
	Finally, finite generation of $\Gamma$, finiteness of each $S_i$, and compact generation of $D$
	turn out to be equivalent. 
\end{enumerate}

\section{Generalities on lattices}

For later use we collect here some lemmas about lattices in locally compact groups.

\begin{lemma}[{\cite{Raghunathan:book}*{Theorem~1.13 on p.~23}}]\label{L:reducing}
Let $G$ be a locally compact second countable group and $N\normal G$ a normal closed subgroup. 
The projection of a lattice $\Gamma<G$ to $G/N$ is discrete if and only if $\Gamma\cap N$ is a lattice in $N$. 
If so, then the projection of $\Gamma$ to $G/N$ is also a lattice. 
\end{lemma}

\begin{lemma}\label{L:intermediate}\cite{Raghunathan:book}*{Lemma~1.6 on p.~20}
	Let $G$ be a locally compact group, and $H<L<G$ closed subgroups. 
	Then $H<G$ has finite covolume if and only if both $L<G$ and $H<L$ have finite covolumes.
\end{lemma}

The following easy lemma is well known (cf.~\cite{Caprace+Monod:kacmoody}*{2.C}).

\begin{lemma}\label{L:open}
	Let $\Gamma<G$ be a lattice in a locally compact second countable group. Let $U<G$ be an open subgroup.
	Then $\Gamma\cap U$ is a lattice in $U$.
\end{lemma}

\section{Around Margulis' super-rigidity for commensurators}
\label{sec: Margulis commensurator rigidity}

Throughout, let $K$ be a number field and $\bfH$ a connected almost simple $K$-algebraic group. We also retain the notation from Section~\ref{sub: intro-notation}. 

Let $\nu$ be a place of $K$. Whenever we consider 
the $\calO$-points or $\calO_\nu$-points of $\bfH$ for some place $\nu$ we pick an embedding $\bfH<\mathbf{GL}_n$ in advance and set 
$\bfH(\calO)=\mathbf{GL}_n(\calO)\cap \bfH(K)$ and similarly for $\calO_\nu$. Let $\tilde \bfH$ be the simply connected cover of $\bfH$ (in the algebraic sense) and $i\colon\tilde \bfH\to \bfH$ be the canonical central isogeny. Since $i$ is defined over $K$ but not necessarily over $\calO_\nu$ for a place $\nu$ of $K$ we do not necessarily have $i(\tilde \bfH(\calO_\nu))\subset\bfH(\calO_\nu)$. But according to the following result there are only finitely many exceptional places. 

\begin{lemma}[\cite{Platonov+Rapinchuk}*{Lemma~6.6 on p.~295}]\label{lem: defined over integral points}
	For all but finitely many finite places $\nu$, the central isogeny $i$ is defined over $\calO_\nu$, thus $i(\tilde \bfH(\calO_\nu))\subset\bfH(\calO_\nu)$. 
\end{lemma}

Let $E_1$ be the set of exceptional places from the previous lemma, and let $E_2$ be the of exceptional places described by the next lemma. 

\begin{lemma}[\cite{tits}*{3.9.1 on p.~55}]\label{lem: maximal compact}
For all but finitely many finite places $\nu$, $\bfH(\calO_\nu)<\bfH(K_\nu)$ is a maximal compact subgroup. Similarly for $\tilde \bfH$.  	
\end{lemma}

Let $\nu\not\in E_1\cup E_2$. Then $i^{-1}(\bfH(\calO_\nu))$ is a compact subgroup that contains $\tilde \bfH(\calO_\nu)$. Hence it coincides with it. We record this for later use: 

\begin{corollary}\label{cor: exceptional places}
	For all but finitely many finite places $\nu$, $i^{-1}(\bfH(\calO_\nu))=\tilde\bfH(\calO_\nu)$. 
\end{corollary}

\begin{lemma} \label{lemma:+}
%Let $K$ be a number field and $\mathbf{H}$ be a connected, almost simple $K$-algebraic group.
%Let $\tilde{\mathbf{H}}$ denote the simply connected cover of $\mathbf{H}$ and  
%$i:\tilde{\mathbf{H}}\to \mathbf{H}$ be the canonical central isogeny.
Let $I\subset \calV^\fin$ be the set of finite places $\nu$ of $K$ such that $\mathbf{H}$ is $K_\nu$-isotropic. Let $J\subset I$. Assume $\mathbf{H}_{\calV^\infty}$ is not compact. Then: 
\begin{enumerate}
\item The image of $i:\tilde{\mathbf{H}}_J\to\mathbf{H}_J$ equals $\mathbf{H}_J^+$. 
\item $\mathbf{H}_J^+<\mathbf{H}_J$ is closed, and the map $\tilde{\mathbf{H}}_J\to\mathbf{H}_J^+$ is open. 
\item $\mathbf{H}_J^+$ is an open subgroup of $\mathbf{H}_J^\#$.
\end{enumerate}
\end{lemma}

\begin{proof}\hfill

%Notice that once we prove that the image of $i$ in $\mathbf{H}_J$ is $\mathbf{H}_J^+$ and that it is closed,
%the fact that the map $\tilde{\mathbf{H}}_J \to \mathbf{H}_J^+$ is open
%will follow by a standard Baire category argument.
%We proceed to prove that the image of $i$ in $\mathbf{H}_J$ is $\mathbf{H}_J^+$ and that it is closed.
(1) For every $\nu\in J$ by \cite[Corollary~I.2.3.1(a)]{Margulis:book} and \cite[Corollary~I.1.5.5]{Margulis:book},
$i(\tilde{\mathbf{H}}(K_\nu))=i(\tilde{\mathbf{H}}(K_\nu)^+)=\mathbf{H}(K_\nu)^+$.
By \cite[Corollary~I.2.3.1(b)]{Margulis:book} $\mathbf{H}(K_\nu)^+$ is closed in $\mathbf{H}(K_\nu)$. With Lemma~\ref{lem: defined over integral points} we obtain that $i(\tilde{\mathbf{H}}_J)\subset \mathbf{H}_J^+$. For the reverse inclusion, consider an element $(x_\nu)_{\nu\in J}\in \bfH_J^+$. For each $\nu\in J$ let $y_\nu\in \tilde\bfH(K_\nu)$ be a pre-image of $x_\nu$, so $i(y_\nu)=x_\nu$. Since $x_\nu\in \bfH(\calO_\nu)$ for all but finitely many $\nu$, we also have $y_\nu\in\tilde \bfH(\calO_\nu)$ for all but finitely many $\nu$. In other words, $(y_\nu)_{\nu\in J}\in \tilde\bfH_J$, and we proved that 
$i(\tilde{\mathbf{H}}_J)=\mathbf{H}_J^+$. 

%
%Thus  and it is closed for every finite $J$.
%By {\color{red} reference}
%there exists a finite set $F\subset \calV^\fin$ such that for every finite place $\nu$ which is not in $F$, $\tilde{\mathbf{H}}(\calO_\nu)<\tilde{\mathbf{H}}(K_\nu)$
%is maximal compact.
%Notice that for every $\nu\in \calV^\fin$, $i^{-1}(\mathbf{H}(\calO_\nu))$ is a compact group containing $\tilde{\mathbf{H}}(\calO_\nu)$, so for $\nu\notin F$ they must coincide.
%It follows that 
%\[ \prod_J \mathbf{H}(\calO_\nu)\cap \mathbf{H}_J^+ < i(\prod_{J\cap F} \tilde{\mathbf{H}}(K_\nu) \times \prod_{J-F} \tilde{\mathbf{H}}(\calO_\nu) )
%< i(\tilde{\mathbf{H}}_J). \]
%We get that for every finite set of places $S\subset J$,
%\begin{equation} \label{eq:+}
%\mathbf{H}_J^+ \cap (\mathbf{H}_S\times \prod_{J-S} \mathbf{H}(\calO_\nu)) < i(\tilde{\mathbf{H}}_J).
%\end{equation}
%We conclude that $ \mathbf{H}_J^+<i(\tilde{\mathbf{H}}_J)$. The other inclusion being obvious, we get that the image of $i$ is indeed $\mathbf{H}_J^+$.
%%%%%%
%By \cite{Platonov+Rapinchuk}*{Theorem~7.12} the simply connected group $\tilde{\mathbf{H}}$
%satisfies the strong approximation property.
%In particular, the image of $\tilde{\mathbf{H}}(K)$ in $\tilde{\mathbf{H}}_J$ is dense.
%It follows that the image of $\tilde{\mathbf{H}}(\calO)$ is dense in $\prod_{\nu\in J} \tilde{\mathbf{H}}(\calO_\nu)$.
%%%%%
(2) Note that $\bfH_J$ as a topological space is the colimit of $\bfH_S^+\times\prod_{J-S} \bfH(\calO_\nu)$ running over 
finite subsets $S\subset J$. 
The subgroup $\bfH_J^+<\bfH_J$ is closed since for every finite subset $S\subset J$ the intersection 
\[ \bfH_J^+\cap \Bigl(\bfH_S^+\times\prod_{J-S} \bfH(\calO_\nu)\Bigr)=\mathbf{H}_S^+\times \Bigl(\prod_{J-S} \mathbf{H}(\calO_\nu)\cap \mathbf{H}(K_\nu)^+\Bigr) \] 
is closed being the product of a closed and compact groups. By the open mapping theorem~\cite{bourbaki}*{IX. \S 5.3} we conclude that the map $\tilde{\mathbf{H}}_J\to\mathbf{H}_J^+$ is open.

(3) By \cite{Platonov+Rapinchuk}*{Theorem~7.12} the simply connected group $\tilde{\mathbf{H}}$
satisfies the strong approximation property.
Hence the image of $\tilde{\mathbf{H}}(K)$ in $\tilde{\mathbf{H}}_J$ is dense which implies that $\bfH_J^+$ is a subgroup 
of $\bfH_J^\#$. To show that it is an open subgroup we first consider the case $J=\calV^\fin$:
 
Let $U=\prod_{\nu\in \calV^\fin} \mathbf{H}(\calO_\nu)$; it is a compact open subgroup of $\bfH_{\calV^\fin}$. 
By Corollary~\ref{cor: exceptional places} there is a finite subset $E\subset\calV^\fin$ of exceptional places and 
compact open subgroups $C_\nu<\tilde\bfH(K_\nu)$ for $\nu\in E$ so that 
\[ i^{-1}(U)=\prod_{\nu\in\calV^\fin\bs E}\tilde\bfH(\calO_\nu)\times \prod_{\nu\in E} C_\nu\]
and $i(\tilde\bfH(\calO_\nu))=\bfH(\calO_\nu)$ if $\nu\in\calV^\fin\bs E$. 
In particular, $i^{-1}(U)$ and $\prod_{\nu\in \calV^\fin}\tilde\bfH(\calO_\nu)$ are commensurable.
%, and 
%\[ i\bigl(i^{-1}(U)\bigr)<\bfH_{\calV^\fin}^+\cap U\] 
%is a finite index subgroup since each $i(C_\nu)<\bfH(\calO_\nu)$, $\nu\in E$, is open and compact 
%by~\cite{Margulis:book}*{2.3.1 (c) on p.~52}. 
Further, the groups $i(\tilde{\mathbf{H}}(\calO))$ and $\mathbf{H}(\calO)$ are commensurable within $\bfH(K)$ 
by~\cite[Corollary~I.3.2.9]{Margulis:book}, and $\mathbf{H}(\calO)$ is dense in $\mathbf{H}_{\calV^\fin}^\#\cap \prod_{\nu\in \calV^\fin} \mathbf{H}(\calO_\nu)$. We prove that $i(\tilde{\mathbf{H}}_{\calV^\fin})=\bfH_{\calV^\fin}^+$
is open in $\mathbf{H}_{\calV^\fin}^\#$ by proving that the closed subgroup (see (2)) 
$\bfH_{\calV^\fin}^+\cap U<\mathbf{H}_{\calV^\fin}^\#\cap U$ has finite index, thus is open. This follows from the following 
commutative diagram which collects the information obtained so far. 

\begin{equation*}
	\begin{tikzcd}
	\tilde\bfH (\calO)\arrow[r, hook]\arrow[dd,twoheadrightarrow] &\prod_{\nu\in\calV^\fin}\tilde\bfH(\calO_\nu)\arrow[d, leftrightsquigarrow, ,"\text{commen.}"]\\
	&     i^{-1}(U)\arrow[r, hook]\arrow[d, "i"] & \tilde\bfH_{\calV^\fin}\arrow[d, "i"]\\
	i(\tilde\bfH(\calO))\arrow[d, leftrightsquigarrow, ,"\text{commen.}"] & \bfH_{\calV^\fin}^+\cap U\arrow[d,hook]\arrow[r,hook]& \bfH_{\calV^\fin}^+\arrow[d,hook]\\
	\bfH(\calO)\arrow[r,hook,"\text{dense}"] & \bfH_{\calV^\fin}^\#\cap U\arrow[r,hook] & \bfH_{\calV^\fin}^\#
	\end{tikzcd}
\end{equation*}

The diagram shows that there is a subgroup $L<i(\tilde\bfH(\calO))\cap \bfH(\calO)<\bfH(K)$ that has finite index 
in $\bfH(\calO)$ such that the induced homomorphism 
\[ \bfH(\calO)/L\rightarrow \bfH_{\calV^\fin}^\#\cap U/\bfH_{\calV^\fin}^+\cap U\] 
has dense image. Since the left hand side is finite, it is surjective and so the right hand side is finite too. 

Next we show the openness statement for general $J$. 
Note that the projection $\mathbf{H}_{\calV^\fin}\to \mathbf{H}_J$, which is locally given as the proper homomorphism 
$ \prod_{\nu\in \calV^\fin} \mathbf{H}(\calO_\nu) \to  \prod_{\nu\in J} \mathbf{H}(\calO_\nu)$, sends 
closed groups to closed groups.
It follows that the image of $\mathbf{H}_{\calV^\fin}^\#$ is $\mathbf{H}_J^\#$, and $p:\mathbf{H}_{\calV^\fin}^\#\to \mathbf{H}_J^\#$
is open by the open mapping theorem~\cite{bourbaki}*{IX. \S 5.3}.
We conclude that of $\bfH_J^+=i(\tilde\bfH_J)=p(i(\tilde\bfH_{\calV^\fin}))<\mathbf{H}_J^\#$ is open.
\end{proof}

\begin{lemma} \label{thm:sa}
Let $K$ be a number field and $\mathbf{H}$ be a connected, almost simple $K$-algebraic group.
Assume $\mathbf{H}_{\calV^\infty}$ is not compact.
Let $I\subset \calV^\fin$ be the set of non-archimedean places $\nu$ such that $\mathbf{H}$ is $K_\nu$-isotropic.
Let $\Sigma$ be a subgroup of $\mathbf{H}(K)$ which contains a finite index subgroup of $\mathbf{H}(\calO)$.
Then the closure of the image of $\Sigma$ in $\mathbf{H}_I$ contains an open subgroup of $\mathbf{H}_I^+$.
\end{lemma}

\begin{proof}
We let $i:\tilde{\mathbf{H}}\to \mathbf{H}$ be the canonical central isogeny, where $\tilde{\mathbf{H}}$ denotes the simply connected cover of $\mathbf{H}$.
By \cite{Platonov+Rapinchuk}*{Theorem~7.12} the simply connected group $\tilde{\mathbf{H}}$
satisfies the strong approximation property.
In particular, the image of $\tilde{\mathbf{H}}(K)$ in $\tilde{\mathbf{H}}_I$ is dense.
It follows that the image of $\tilde{\mathbf{H}}(\calO)$ is dense in $\prod_{\nu\in I} \tilde{\mathbf{H}}(\calO_\nu)$.
Therefore the image of any finite index subgroup of $\tilde{\mathbf{H}}(\calO)$ has an open closure in $\tilde{\mathbf{H}}_I$, and so is the closure of the image of $i^{-1}(\Sigma)$.
By Lemma~\ref{lemma:+} the image of $\tilde{\mathbf{H}}_I\to \mathbf{H}_I$ is $\mathbf{H}_I^+$ and the map $\tilde{\mathbf{H}}_I\to\mathbf{H}_I^+$
is open, hence we are done by the commutativity of the obvious diagram 
	\[\begin{tikzcd}
     	 \tilde{\mathbf{H}}(K) \arrow[d]\arrow[r] &  \tilde{\mathbf{H}}_I\arrow[d]\\
		\mathbf{H}(K)\arrow[r] & \mathbf{H}_I
\end{tikzcd}\qedhere\]
\end{proof}

\begin{theorem} \label{thm:algsetting}
Let $K$ be a number field and $\mathbf{H}$ be a connected, adjoint, almost simple $K$-algebraic group.
Assume $\mathbf{H}_{\calV^\infty}$ is not compact.
Let $\Sigma$ be a subgroup of $\mathbf{H}(K)$ which contains a finite index subgroup of $\mathbf{H}(\calO)$.
Let $S$ be the set of places $\nu$ of $K$ such that $\Sigma$ is unbounded in $\mathbf{H}(K_\nu)$.
Then the following hold:
\begin{enumerate}
\item
The closure of the projection of $\Sigma$ to $\mathbf{H}_{S^\fin}$ contains $\mathbf{H}_{S^\fin}^+$
and it has no non-trivial compact normal subgroups.
\item
The closure of the projection of $\Sigma$ to $\mathbf{H}_{\calV-S}$ is compact.
\item
A finite index subgroup of $\Sigma$ is contained in $\mathbf{H}(\calO[S])$.
\end{enumerate}
\end{theorem}

\begin{proof}
In the proof we will use $L_J$ to denote the closure of the diagonal embedding of $\Sigma$ in $\mathbf{H}_J$
for $J\subset \calV$.
We will also use $L_\nu=L_{\{\nu\}}$.

(1) For $\nu\in S$, $L_\nu$ is unbounded.
By \cite[Corollary~I.2.3.1(c2)]{Margulis:book}, $\mathbf{H}(K_\nu)/\mathbf{H}(K_\nu)^+$ is finite, 
thus $L_\nu\cap \mathbf{H}(K_\nu)^+$ is unbounded.
By Lemma~\ref{thm:sa}, $L_\nu\cap \mathbf{H}(K_\nu)^+$ is open in $\mathbf{H}(K_\nu)^+$.
It follows from \cite[Theorem~T]{Prasad} that $\mathbf{H}(K_\nu)^+<L_\nu$.
The group generated by the groups $\mathbf{H}(K_\nu)^+$ for $\nu\in S$ is dense in $\mathbf{H}_S^+$.
Since  $L_S$ contains an open subgroup of $\mathbf{H}_S^+$ by Lemma~\ref{thm:sa},
it actually contains $\mathbf{H}_S^+$,
that is $\mathbf{H}_S^+<L_S<\mathbf{H}_S$.
Assume $N<L_S$ is a non-trivial compact normal subgroup.
Then there exists $\nu\in S$ such that the image of $N$, $\bar{N}$, is non trivial under the projection $\mathbf{H}_S\to \mathbf{H}(K_\nu)$.
By \cite[Theorem~I.1.5.6(i)]{Margulis:book}, either $\mathbf{H}(K_\nu)^+<\bar{N}$ or $\bar{N}$ is central.
Both cases are impossible, the first by the compactness of $N$ and the second by $\mathbf{H}$ being adjoint.

(2) First we define two finite exceptional sets $F_1,F_2\subset \calV^\fin$ of finite places. 
We let $i:\tilde{\mathbf{H}}\to \mathbf{H}$ again denote the canonical central isogeny of the simply connected form $\tilde{\mathbf{H}}$ of $\mathbf{H}$. 

According to Lemma~\ref{lem: maximal compact} and Corollary~\ref{cor: exceptional places} there is a finite set $F_1$ of finite places 
such that for $\nu\in\calV^\fin-F_1$ we have that 
$i^{-1}(\bfH(\calO_\nu))=\tilde\bfH(\calO_\nu)$ and both 
$\tilde\bfH(\calO_\nu)$ and $\bfH(\calO_\nu)$ are maximal compact 
subgroups in $\tilde\bfH(K_\nu)$ and $\bfH(K_\nu)$, respectively. 
For $\nu\in \calV^\fin-F_1$, $\mathbf{H}(\calO_\nu)$ is the unique compact open subgroup of 
$\mathbf{H}(K_\nu)$ that contains
$i(\tilde{\mathbf{H}}(\calO_\nu))$.
To see this consider the Bruhat-Tits building of $\mathbf{H}(K_\nu)$ with its $\tilde{\mathbf{H}}(K_\nu)$-action via $i$. The subgroup 
$\tilde{\mathbf{H}}(\calO_\nu)$ has a unique fixed point, hence it is contained in a unique maximal compact of $\mathbf{H}(K_\nu)$.

Next we define $F_2$.
By \cite{Platonov+Rapinchuk}*{Theorem~7.12} the simply connected group $\tilde{\mathbf{H}}$
satisfies the strong approximation property.
In particular, the image of $\tilde{\mathbf{H}}(K)$ in $\tilde{\mathbf{H}}_{\calV^\fin}$ is dense.
It follows that the image of $\tilde{\mathbf{H}}(\calO)$ is dense in $\prod_{\nu\in \calV^\fin} \tilde{\mathbf{H}}(\calO_\nu)$.
Therefore the image of any finite index subgroup of $\tilde{\mathbf{H}}(\calO)$ has an open closure in $\tilde{\mathbf{H}}_{\calV^\fin}$.  Hence $\overline{i^{-1}(\Sigma)}<\tilde{\mathbf{H}}_{\calV^\fin}$ is open. 
It follows that $\overline{i^{-1}(\Sigma)}$ contains a finite index subgroup of the compact group 
$\prod_{\calV^\fin}\tilde{\mathbf{H}}(\calO_\nu)$. 
Hence there exists a finite set $F_2\subset \calV^\fin$ 
such that for every $\nu\in \calV^\fin-F_2$, $\tilde{\mathbf{H}}(\calO_\nu)<\overline{i^{-1}(\Sigma)}$.

We now fix $\nu\in \calV^\fin-S-F_1-F_2$ and claim that $L_\nu<\mathbf{H}(\calO_\nu)$.
Observe first that $i^{-1}(L_\nu)=\tilde{\mathbf{H}}(\calO_\nu)$.
Indeed, $i^{-1}(L_\nu)$ is a compact group because of $\nu\notin S$  that contains 
$\tilde{\mathbf{H}}(\calO_\nu)$ because of $\nu\notin F_2$. 
Hence it coincides with it because of $\nu\notin F_1$.
It follows $i(\tilde{\mathbf{H}}(\calO_\nu))<L_\nu$, and 
the claim follows by the fact that for $\nu\in \calV^\fin-F_1$, $\mathbf{H}(\calO_\nu)$
is the unique maximal compact subgroup which contains $i(\tilde{\mathbf{H}}(\calO_\nu))$.

We conclude that $L_{V-S}$ is contained in the compact group
\[ 
	\Bigl(\prod_{(\calV-S)\cap (\calV^\infty\cup F_1\cup F_2)} L_\nu\Bigr) 
	\times \Bigl(\prod_{\calV-S-(\calV^\infty\cup F_1\cup F_2)} \mathbf{H}(\calO_\nu)\Bigr),
\]
hence it is compact, proving our second point.

(3) By compactness a finite index subgroup of $L_{\calV^\fin-S}$ is contained $\prod_{\calV^\fin-S} \mathbf{H}(\calO_\nu)$,
hence its preimage in $\Sigma$ is contained in $\mathbf{H}(\calO[S])$.
\end{proof}

Note that in the setting of Theorem~\ref{thm:algsetting}, $\mathbf{H}_{S^\infty}$ is a center-free, semisimple real Lie group without compact factors,
the image of $\mathbf{H}(\calO)$ in $\mathbf{H}_{S^\infty}$ under the diagonal embedding is an irreducible lattice 
which is commensurated by the image of $\Sigma$.
Moreover, $S^\fin\neq \emptyset$ if and only if $\mathbf{H}(\calO)$ is of infinite index in $\Sigma$.
The following is a sort of a converse to this situation.

\begin{theorem}[Commensurator Super-Rigidity]\label{thm: margulis lemma}\hfill{}\\
Let $H$ be a connected, center-free, semisimple real Lie group without compact factors.
Let $\Delta<H$ be an irreducible lattice.
	Let $\Lambda$ be a group such 
	that $\Delta<\Lambda<\Commen_{H}(\Delta)$ and $[\Lambda:\Delta]=\infty$. Then
there exist 
a number field $K$, 
a set of places $S$,
a connected, adjoint, almost simple $K$-algebraic group $\mathbf{H}$ 
and a subgroup $\Sigma<\mathbf{H}(K)$ and an isomorphism of topological groups
$\phi:\mathbf{H}_{S^\infty}^0 \xrightarrow{\simeq} H$ so that
\begin{enumerate}
\item
	$\phi^{-1}(\Lambda)$ equals the diagonal image of $\Sigma$ in $\mathbf{H}_{S^\infty}$,
\item	 
	$\phi^{-1}(\Delta)$ in $\Sigma$ is commensurated with $\mathbf{H}(\calO)$.
\item 
	a finite index subgroup of $\Sigma$ is contained in $\mathbf{H}(\calO[S])$.
\item
The closure of the projection of $\Sigma$ to $\mathbf{H}_{S^\fin}$ contains $\mathbf{H}_{S^\fin}^+$
and it has no non-trivial compact normal subgroups.
\item $S$ is compatible with $\mathbf{H}$ (see Definition~\ref{def:compatible}).
\end{enumerate}
\end{theorem}

\begin{proof}%[proof of Theorem~\ref{thm: margulis lemma}]
Let $\mathbf{G}$ be the Zariski closure of the image of $H$ in its adjoint representation. 
Then $\mathbf{G}$ is a real algebraic group and $\Ad(H)=\mathbf{G}(\bbR)^0$ (see the beginning of Section~6 of~\cite[IX]{Margulis:book}).
By the center-free assumption on $H$, $H\cong \Ad(H)$. Thus we may and will identify $H$ with $\mathbf{G}(\bbR)^0$.

We apply \cite[Theorem IX.6.5]{Margulis:book}
and get that $\Delta$ is arithmetic in $\mathbf{G}$.
For the definition of arithmeticity see~\cite[Definitions~IX.1.4]{Margulis:book}.
According to this definition there exists a number field $K$, a connected non-commutative absolutely almost simple $K$-algebraic group $\tilde{\mathbf{H}}$, a finite set of places $T\subset \calV$ containing all archimedean places $\nu$ for which $\tilde{\mathbf{H}}$ is $K_\nu$-isotropic
and a continuous homomorphism $\tilde{\phi}\colon\tilde{\mathbf{H}}_T \to \mathbf{G}(\bbR)$
such that $\tilde{\phi}(\tilde{\mathbf{H}}(\calO[T]))$ is commensurable with $\Delta$, where $\calO$ is the ring of integers in $K$.
By \cite[Remark~IX.1.6(i)]{Margulis:book} we may and will assume that $\tilde{\mathbf{H}}$ is simply 
connected (justifying our choice of notation). 
%We denote by $\mathbf{H}$ its adjoint form.

Working through the intricate definition of the class $\Psi$ appearing in \cite[Definitions~IX.1.4]{Margulis:book}, which is given in \cite[IX.1.3]{Margulis:book},
and taking into account that $\mathbf{G}$ is an $\bbR$-group, we get that $T\subset \calV^{\infty}$.
In particular, $\calO[T]=\calO$
and $\tilde{\phi}(\tilde{\mathbf{H}}(\calO))$ is commensurable with $\Delta$.
By \cite[Remark~IX.1.3(ii)]{Margulis:book}, $\Ker(\tilde{\phi})$ is finite and $H=\mathbf{G}(\bbR)^0<\tilde{\phi}(\tilde{\mathbf{H}}_T)$.
The facts above are also explained in \cite[IX.1.5(*)]{Margulis:book}.

Denoting by $I\subset \calV^{\infty}$ the set of archimedean  places $\nu$ such that $\tilde{\mathbf{H}}$ is $K_\nu$-isotropic,
we get $I\subset T \subset \calV^{\infty}$.
We denote by $\mathbf{H}$ the adjoint form of $\tilde{\mathbf{H}}$.
Let $p$ be the composition of the natural projection $\tilde{\mathbf{H}}_T\to \tilde{\mathbf{H}}_I$ and the isogeny $\tilde{\mathbf{H}}_I\to \mathbf{H}_I$.
By the choice of $I$, $p$ has a compact kernel.
By \cite[Theorem I.1.2.3.1(a')]{Margulis:book}, $\tilde{\mathbf{H}}_I=\tilde{\mathbf{H}}_I^+$
and by \cite[Theorem I.1.5.5]{Margulis:book}
its image in $\mathbf{H}_I$ is $\mathbf{H}_I^+$, which equals $\mathbf{H}_I^0$ by
\cite[Theorem I.1.2.3.1(a)2)]{Margulis:book}.
Note that $H=\mathbf{G}(\bbR)^0=\mathbf{G}(\bbR)^+$ has no non-trivial compact normal subgroups, by assumption.
It follows from \cite[Theorem I.1.5.6(i)]{Margulis:book} that also $\tilde{\phi}(\tilde{\mathbf{H}}_T)$ has no normal subgroup,
as it contains $\mathbf{G}(\bbR)^+$.
From the compactness of $\Ker(p)$ we conclude that $\tilde{\phi}$ factors through the image of $p$, namely $\mathbf{H}_I^0$.
We denote the map thus obtained by $\phi$.
Since the image of $\phi$ is a connected subgroup of $\mathbf{G}(\bbR)$ containing $H$,
it must be $H$,
and since $\mathbf{H}_I^0$ has no compact normal subgroups, we may view
$\phi$ as a continuous isomorphism, $\phi\colon \mathbf{H}_I^0 \cong H$.
By a standard Baire category argument $\phi$ is also a homeomorphism.
In particular, $I\neq\emptyset$.

By \cite[Corollary~I.3.2.9]{Margulis:book} the image of $\tilde{\mathbf{H}}(\calO)$ is commensurated with $\mathbf{H}(\calO)$ in $\mathbf{H}(K)$,
which we now embed diagonally in $\mathbf{H}_I$.
Since $\tilde{\phi}(\mathbf{H}(\calO))$ is commensurated with $\Delta$ in $H$ we conclude that $\phi^{-1}(\Delta)$ is commensurated with $\mathbf{H}(\calO)$ in $\mathbf{H}_I$.

Note that for every $\nu\in I$, $\mathbf{H}(K_\nu)$ is not compact.
By Theorem~\ref{thm:appcomm} the commensurator of $\mathbf{H}(\calO)$ in $\mathbf{H}_I$ is $\mathbf{H}(K)$.
Thus, $\phi^{-1}(\Lambda)$ is contained in the diagonal image of $\mathbf{H}(K)$.
We denote by $\Sigma<\mathbf{H}(K)$ its preimage.
We can now consider the preimage of $\Delta$ in $\Sigma$ and we conclude that it is commensurated with $\mathbf{H}(\calO)$.
We are now in the setting of Theorem~\ref{thm:algsetting}, and everything else follows. 
\end{proof}

\section{Arithmetically constructed lattices}\label{sec: arithmetic lattices}

Throughout, we retain the notation of Section~\ref{sub: intro-notation} and 
fix number fields $K_1,\dots,K_n$, connected, non-commutative, absolutely simple $K_i$-groups $\mathbf{H}_i$, 
and (possibly infinite) compatible 
subsets $S_i\subset \calV_i$ of places of $K_i$ (see Definition~\ref{def:compatible}).
By convention, unspecified products are to be taken from $1$ to $n$, e.g by $\prod K_i$ we mean $\prod_{i=1}^n K_i$.

\begin{lemma} \label{lem:acl}
$\prod (\mathbf{H}_i)_{S_i^\fin}^\#$ contains $\prod (\mathbf{H}_i)_{S_i^\fin}^+$ as an open normal subgroup, 
and the discrete group $\prod (\mathbf{H}_i)_{S_i^\fin}^\#/\prod (\mathbf{H}_i)_{S_i^\fin}^+$ is a torsion abelian group.
\end{lemma}

\begin{proof}
By Theorem~\ref{thm:algsetting}, $\prod (\mathbf{H}_i)_{S_i^\fin}^\#$ contains $\prod (\mathbf{H}_i)_{S_i^\fin}^+$. Normality is proved in~\cite{Margulis:book}*{I.2.3.1 (b)}. 
Openness follows from Lemma~\ref{lemma:+}.
The quotient is a torsion abelian group by \cite[I.2.3.1(c)]{Margulis:book}.
\end{proof}

\begin{proof}[Proof of {Proposition~\ref{ex:lattices}}]
By Lemma~\ref{lem:acl} we are left to show that 
\[ \Gamma_A < \prod (\mathbf{H}_i)_{S_i^\infty}\times D_A \]
is indeed a lattice.
By reduction theory, 
\[ \prod \mathbf{H}_i(\calO_i[S_i]) < \prod (\mathbf{H}_i)_{S_i} = \prod (\mathbf{H}_i)_{S_i^\infty}\times \prod (\mathbf{H}_i)_{S_i^\fin} \]
is  a lattice. By Lemma~\ref{L:intermediate}, $\prod \mathbf{H}_i(\calO_i[S_i])$ is a lattice in $\prod (\mathbf{H}_i)_{S_i^\infty}\times \prod (\mathbf{H}_i)_{S_i^\fin}^\#$. 
The claim follows from Lemma~\ref{L:open} and the openness statement 
in Lemma~\ref{lem:acl}. 
\end{proof}

\begin{lemma} \label{lem:subthmcore}
Let $\Gamma$ be a subgroup of $\prod \mathbf{H}_i(K_i)$
and $D$ be the closure of the diagonal embedding of $\Gamma$ to $\prod (\mathbf{H}_i)_{S_i^\fin}$. We assume: 
\begin{enumerate}
	\item $\Gamma$ contains a finite index subgroup of $\prod \mathbf{H}_i(\calO_i)$. 
	\item A finite index subgroup of $\Gamma$ is contained in $\prod \mathbf{H}_i(\calO_i[S_i])$. 
	\item $(\mathbf{H}_i)_{S_i^\fin}^+$ is contained in the projection of $D$ to $(\mathbf{H}_i)_{S_i^\fin}$ for each $i\in\{1,\dots,n\}$. 
\end{enumerate}
Then $\Gamma$ is an arithmetically constructed lattice.
\end{lemma}

\begin{proof}
We first claim that $\prod (\mathbf{H}_i)_{S^\fin_i}^+<D$.
Denote by $\tilde{D}$ the preimage of $D$ in $\prod (\tilde{\mathbf{H}}_i)_{S^\fin_i}$.
Then $\tilde{D}$ is open since it contains a finite index subgroup of the image of $\prod \tilde{\mathbf{H}}_i(\calO_i)$
which is dense in an open subgroup by strong approximation~\cite{Platonov+Rapinchuk}*{Theorem~7.12}.
For each $i$ and $\nu\in S_i^\fin$, $\tilde{D}\cap \tilde{\mathbf{H}}_i(K_\nu)$ is open in $\tilde{\mathbf{H}}_i(K_\nu)$ and normal in $\tilde{D}$. Hence it is also normal in $\tilde{\mathbf{H}}_i(K_\nu)$ as $\tilde{D}\to (\tilde{\mathbf{H}}_i)_{S^\fin_i} \to \tilde{\mathbf{H}}_i(K_\nu)$
is surjective by the assumption that the projection $D\to  (\mathbf{H}_i)_{S^\fin_i}$ contains $(\mathbf{H}_i)_{S^\fin_i}^+$.
By \cite[Theorem~I.1.5.6(i)]{Margulis:book} we get that $\tilde{\mathbf{H}}_i(K_\nu)<\tilde{D}$. 
The group generated by these is dense. Hence $\tilde{D}=\prod (\tilde{\mathbf{H}}_i)_{S^\fin_i}$. 
By Lemma~\ref{lemma:+} we conclude that $\prod (\mathbf{H}_i)_{S^\fin_i}^+<D$.

Denote $\Gamma_0=\Gamma\cap \prod \mathbf{H}_i(\calO_i[S_i])$
and let $D_0$ be the closure of the projection of $\Gamma_0$ to $\prod (\mathbf{H}_i)_{S_i^\fin}$.
Then $\Gamma_0<\Gamma$ and $D_0<D$ are of finite index.

We claim that the group $\prod (\mathbf{H}_i)_{S^\fin_i}^+$ has no proper closed finite index subgroup.
Indeed, for every $i$ and $\nu\in S^\fin_i$, by \cite[Corollary~I.1.5.7]{Margulis:book}, the groups $\mathbf{H}_i(K_\nu)^+$ has no proper finite index subgroup, and the group generated by these is dense in $\prod (\mathbf{H}_i)_{S^\fin_i}^+$.
We conclude that $\prod (\mathbf{H}_i)_{S_i}^+<D_0$.

By Lemma~\ref{lem:acl} $\prod (\mathbf{H}_i)_{S_i^\fin}^\#/\prod (\mathbf{H}_i)_{S_i^\fin}^+$ is discrete.
Since the image of $\Gamma_0$ in $D_0$ is dense, they have the same image in $\prod (\mathbf{H}_i)_{S_i^\fin}^\#/\prod (\mathbf{H}_i)_{S_i^\fin}^+$,
which we denote $A$. With the notation of Proposition~\ref{ex:lattices}, clearly $D_0=D_A$ and $\Gamma_0=\Gamma_A$. So $\Gamma$ is arithmetically constructed. 
\end{proof}

\section{Representations of commensurable pairs}
\label{sub:reps-of-comm}
Recall that two subgroups of an (abstract) group are \emph{commensurable} if 
their intersection has finite index in both subgroups. 
A subgroup $H$ of a group $\Lambda$ is \emph{commensurated} if 
$H$ and any conjugate $\lambda H \lambda^{-1}$ are commensurable. 
Commensurability is an equivalence relation. 
An \emph{invariant commensurability class} in a group is a commensurability equivalence class that is closed under conjugation. 

\begin{example}\label{E:clompact-in-tdlc}
	Let $G$ be a topological group. 
	Then the class $\mathcal{K}_G$ of compact open subgroups of $G$ is an invariant commensurability class.
\end{example}
A pair $(\Lambda, C)$ consisting of an abstract group $\Lambda$ and 
an invariant commensurability class $C$ in $\Lambda$ is called a 
\emph{commensurability pair}. 
If $(\Lambda, C)$ is an invariant commensurability pair then every $\Delta\in C$ is a commensurated subgroup in $\Lambda$;
conversely, the commensurability class $C$ of any commensurated subgroup $\Delta<\Lambda$ defines an invariant commensurability pair $(\Lambda,C)$.
Note also that under a group homomorphism, the preimage of a commensurated  subgroup is commensurated. 

\begin{definition}\label{def: tdlc representation}
	A {\em tdlc representation} of a commensurability pair $(\Lambda,C)$ is a tdlc group $G$ together with a homomorphism
	$\phi:\Lambda\to G$ such that $\phi^{-1}(K)\in C$ for some (hence any) $K\in\mathcal{K}_G$. 
	If, in addition, $\im(\phi)$ is dense in $G$ 
	we call it a \emph{dense tdlc representation}. 

	A \emph{morphism} of tdlc representations $\phi\colon\Lambda\to G$ and $\phi'\colon\Lambda\to G'$ 
	is a continuous homomorphism $\psi:G\to G'$ such that $\psi\circ \phi=\phi'$.
\end{definition}

% Thus example~\ref{E:clompact-in-tdlc} gives
% \begin{lemma}\label{L:a-normal}
% Let $G$ be a locally compact group and let $f\colon\Lambda \to G$ be a homomorphism. 
% Let $U<G$ be a compact open subgroup. Then $f^{-1}(U)$ is commensurated by $\Lambda$.
% \end{lemma}
% %
% We will show that this is the source of all examples of tdlc representations.

%\begin{proof}
%	For every $\gamma\in\Gamma$ the homomorphism $f$ induces an isomorphism 
%	\[f^{-1}(U)/\bigl(\gamma^{-1}f^{-1}(U)\gamma\cap f^{-1}(U)\bigr)\xrightarrow{\cong} U/\bigl(f(\gamma)^{-1}Uf(\gamma)\cap U\bigr) .\]
%   Since $f(\gamma)^{-1}Uf(\gamma)\cap U$ is an open subgroup in the compact 
%   group $U$, its index is finite. It follows that $\Gamma\cap f^{-1}(U)$ is 
%   commensurated by $\Gamma$. 
%\end{proof}

Let $\phi:\Lambda\to G$ be a tdlc representation of a commensurability pair $(\Lambda, C)$. 
For $\Delta\in C$, the closure of $\phi(\Delta)$ is compact in $G$, and it is also open, thus $\overline{\phi(\Delta)}\in\mathcal{K}_G$, 
provided $\phi$ is dense. 
Further, for any open group $H<G$, $H=\overline{\phi(\phi^{-1}(H))}$. We summarise: 

\begin{lemma}
For a dense tdlc representation $\phi\colon \Lambda\to G$ of a commensurability pair $(\Lambda, C)$, $C$ is the commensurability class of any preimage $\phi^{-1}(K)$ with $K\in\mathcal{K}_G$ and $\{\overline{\phi(\Delta)}\mid \Delta\in C\}=\mathcal{K}_G$. 
\end{lemma}

Because of this lemma we may and will drop $C$ from the notation when dealing with dense tdlc representations.

%
%\begin{lemma}
%Let $\phi:\Lambda\to G$ be a dense representation.
%Let $H<G$ be an open group with $\phi^{-1}(H)\in C$.
%Then $H$ is compact.
%\end{lemma}

%\begin{proof}
%For $K<H$ open and compact $|H/K|=|\phi^{H}/\phi{-1}(K)|<\infty$.
%\end{proof}

\begin{lemma} \label{lem:closedimage}
	Let $\psi:G\to G'$ be a morphism of dense tdlc representations from $\phi:\Lambda\to G$ to $\phi':\Lambda\to G'$.
	Then $\psi$ is surjective with compact kernel.
\end{lemma}

\begin{proof}
Let $K'\in\mathcal{K}_{G'}$. Let 
$\Delta=\phi'^{-1}(K')$ and $K=\overline{\phi(\Delta)}$.
By the previous lemma $K\in\mathcal{K}_G$. We have 
\[ 
	\psi(K)=\overline{\psi(K)}=\overline{\psi(\phi(\Delta))}=\overline{\phi'(\Delta)}=K' 
\]
and
\[ 
	\psi^{-1}(K')=\overline{\phi(\phi^{-1}( \psi^{-1}(K')))}=\overline{\phi(\phi'^{-1}(K'))}=\overline{\phi(\Delta)}=K. 
\]
In particular, $\Ker(\psi)<K$ and
\[
	G'=K'\phi'(\Lambda)=\psi(K)\psi(\phi(\Lambda))=\psi(K\phi(\Lambda))=\psi(G). \qedhere
\]
\end{proof}

The following theorem is well known (however we believe our easy proof below is novel).
It will not be used in these notes, it is only given for context.

\begin{theorem}
	The category of tdlc representations of the pair $(\Lambda,C)$ has an initial object, which is a dense representation.
\end{theorem}

\begin{proof}
%[Sketch of a proof]
%{\color{red} check this. if there is a flaw replace it by a traditional one.}
The collection of isomorphism classes of dense tdlc representations of $(\Lambda, C)$ forms a set $\Sigma$.
Fix some $\Delta\in C$. For each dense representation $\phi\colon \Lambda\to G$ in $\Sigma$ we pick the closure of the image of $\Delta$ in $G$ as a distinguished subgroup of $G$. 
Then we take the restricted product inside $\prod_{(\phi\colon\Lambda\to G)\in\Sigma} G$ with respect to these distinguished subgroups. Finally, we 
take the closure $\hat{G}$ of the image of $\Lambda$ under the diagonal homomorphism. We claim that the resulting representation $\hat{\phi}:\Lambda\to \hat{G}$ is an initial object in the category of tdlc representations. 

Let $\phi\colon\Lambda\to G$ be a tdlc representation. By construction, $\hat{G}$ admits a projection onto an isomorphic copy of $\overline{\phi(\Lambda)}$.
Composing this projection with the inverse of the alluded isomorphism and the inclusion of $\overline{\phi(\Lambda)}$ in $G$ we get $\psi\colon\hat{G}\to G$,
which is easily seen to be a morphism of representations.
The uniqueness of this morphism is clear by the density of $\hat{\phi}$.
\end{proof}

%We will denote ``the" initial object by $\hat{\phi}:\Lambda\to\hat{G}$.
In the next theorem we consider the category of dense tdlc representations, which is a full subcategory of the category of tdlc representations.

\begin{theorem} \label{thm:terminal}
	Let $\check{\phi}\colon\Lambda\to \check{G}$ be a dense tdlc representation such that $\check{G}$ has no non-trivial compact normal subgroup.
	Then $\check{\phi}$ is a terminal object in the category of dense tdlc representations.
\end{theorem}

\begin{proof}
Let $\phi\colon\Lambda\to G$ be a dense tdlc representation.
We only need to show the existence of a morphism $G\to \check{G}$ since there is at most one morphism between dense representations. 
We denote $G':=\overline{(\phi\times\check{\phi})(\Lambda)}<G\times \check{G}$ and let $\phi'$ denote the diagonal representation,
with the codomain restricted to $G'$.
%The projection on the second coordinate gives a morphism form $\phi'$ to $\check{\phi}$.
%We argue to show that the projection on the first coordinate forms an isomorphism between $\phi'$ and $\phi$.
%This will prove the theorem.
Let us denote $\psi=\pr_1|_{G'}$.
By Lemma~\ref{lem:closedimage} $\psi$ is surjective with compact kernel.
Since $\Ker(\psi) \lhd \{e\}\times \check{G} \cong \check{G}$ we conclude that $\psi$ is injective.
By the open mapping theorem~\cite{bourbaki}*{IX. \S 5.3} we conclude that $\psi$ is a topological group isomorphism.
It is now easy to see that $\pr_2\circ\psi^{-1}\colon G\to \check{G}$ is a morphism of tdlc representations.
\end{proof}

The following lemma provides an example of a terminal representation appearing in an arithmetic group setting.
It will be useful later.

\begin{lemma} \label{lem:terminal}
We retain the setting and notation described at the start of 
Section~\ref{sec: arithmetic lattices}. In addition we assume that 
\begin{enumerate}
	\item each $\bfH_i$ is adjoint, 
	\item $\Sigma$ contains a finite index subgroup of $\prod \mathbf{H}_i(\calO_i)$, 
	\item a finite index subgroup of $\Sigma$ is contained in $\prod \mathbf{H}_i(\calO_i[S_i])$, 
	\item the closure of the projection $\Sigma\to (\mathbf{H}_i)_{S_i^\fin}$ has no non-trivial compact normal subgroup for each $i\in\{1,\dots,n\}$.
\end{enumerate}
Let $Q<\prod (\mathbf{H}_i)_{S_i^\fin}$ be the closure of the image of $\Sigma$ under the diagonal map, and let $C$ be the commensurability class of $\Sigma\cap \prod \mathbf{H}_i(\calO_i)$ in $\Sigma$. 

Then the embedding $\Sigma\to Q$ is a dense tdlc representation of the commensurability pair $(\Sigma,C)$ which is is a terminal object in the category of dense tdlc representations of $(\Sigma, C)$.
\end{lemma}

\begin{proof}
By Theorem~\ref{thm:appcomm} $\prod \mathbf{H}_i(\calO_i)$ is commensurated in $\prod \mathbf{H}_i(K_i)$. Thus the group 
$\Sigma\cap \prod \mathbf{H}_i(\calO_i)$ is commensurated in $\Sigma$ and $C$ is invariant. For each $i$ let $T_i=\calV_i^\fin-S_i$. 
Observe that 
\[ 
	\prod \mathbf{H}_i(\calO_i[T_i]) \cap \prod \mathbf{H}_i(\calO_i[S_i])=\prod \mathbf{H}_i(\calO_i). 
\]
Thus
\[ 
	\left(\Sigma \cap \prod \mathbf{H}_i(\calO_i[T_i])\right) 
	\cap \left(\Sigma \cap \prod \mathbf{H}_i(\calO_i[S_i])\right)=\Sigma\cap\prod \mathbf{H}_i(\calO_i). 
\]
Since $\Sigma \cap \prod \mathbf{H}_i(\calO_i[S_i])<\Sigma$ is of finite index, 
$\Sigma \cap \prod \mathbf{H}_i(\calO_i[T_i])$ is commensurable with $\Sigma\cap \prod \mathbf{H}_i(\calO_i)$, 
i.e $\Sigma \cap \prod \mathbf{H}_i(\calO_i[T_i])\in C$.

We claim that the embedding $\phi:\Sigma\to \prod (\mathbf{H}_i)_{S_i^\fin}$ is a tdlc representation: 
Consider the compact open subgroup
\[ 
	U=\prod_i \prod_{S_i} \mathbf{H}_i(\calO_i)_\nu)< \prod_i (\mathbf{H}_i)_{S_i^\fin} 
\]
The preimage of $U$ in $\prod \mathbf{H}_i(K_i))$ is exactly $\prod \mathbf{H}_i(\calO_i[T_i])$,
thus the preimage in $\Sigma$ is $\Sigma \cap \prod \mathbf{H}_i(\calO_i[T_i])$, which is in $C$, proving our claim.

Hence $\Sigma\to Q$ is a dense tdlc representation of the commensurability pair $(\Sigma,C)$.
By Theorem~\ref{thm:terminal} this is a terminal object if $Q$ has no non-trivial compact normal subgroup. The latter is due to assumption (4). 
\end{proof}

%%%%%%%%%%%%%%%%%%%%%%%%%

\section{Proof of Theorem ~\ref{T:arithmetic-core}}
	Let $U<D$ be an open compact subgroup. 
	We denote its preimage under $\Gamma\to D$ by $\Gamma|_U$.
	It is commensurated in $\Gamma$ by Lemma~\ref{L:a-normal}.
	By Lemma~\ref{L:open}
	$\Gamma|_U$ is a lattice in $H\times U$ and by Lemma~\ref{L:reducing} its projection to $H$ is a lattice in $H$.
	We denote $\Delta=\pr_H(\Gamma|_U)$ and $\Lambda=\pr_H(\Gamma)$.
	Thus $\Delta<\Lambda<H$, $\Delta$ is a lattice in $H$ which is commensurated by $\Lambda$ and $\Lambda$ is dense in $H$.

 	The lattice $\Delta<H$ might be reducible. The decomposition into irreducible components 
	gives a splitting $H=H_1\times \dots \times H_n$ 
	into $n$ factors and irreducible lattices $\Delta_i<H_i$, $i\in\{1,\ldots, n\}$ such 
	that $\Delta_1\times\cdots\times \Delta_n<\Delta$ is of finite index. 
	%We denote $\Delta=\Delta_1\times\cdots\times \Delta_n$ and notice that it is commensurated in $\Lambda$.
	Let $\Lambda_i=\pr_{H_i}(\Gamma)$, where $\pr_{H_i}$ is the 
	projection $H\times D\to H\to H_i$. Clearly,  
	$\Lambda_i$ is dense in $H_i$, which implies $[\Lambda_i:\Delta_i]=\infty$, and 
	$\Delta_i$ is commensurated by $\Lambda_i$.

	We are in a position to apply Theorem~\ref{thm: margulis lemma} and obtain for each $i$:
	a number field $K_i$, a compatible set of places $S_i$, a connected, adjoint, almost simple $K_i$-algebraic group $\mathbf{H}_i$ 
	and a subgroup $\Sigma_i<\mathbf{H}_i(K_i)$,
	satisfying the following properties:
	\begin{itemize}
	\item
		There exists an isomorphism of topological groups
		$\phi_i:(\mathbf{H}_i)_{S_i^\infty}^0 \xrightarrow{\simeq} H_i$
		such that 
		$\phi_i^{-1}(\Lambda_i)$ equals the diagonal image of $\Sigma_i$ in $(\mathbf{H}_i)_{S_i^\infty}$ and
		the preimage of $\Delta_i$ in $\Sigma_i$ is commensurated with $\mathbf{H}_i(\calO_i)$.
	\item 
		$\Sigma_i$ contains a finite index subgroup of $\mathbf{H}_i(\calO_i)$ and a finite index subgroup of $\Sigma_i$ is contained in $\mathbf{H}_i(\calO_i[S_i])$.
	\item
		The closure of the projection of $\Sigma_i$ to $(\mathbf{H}_i)_{S_i^\fin}$ contains $(\mathbf{H}_i)_{S_i^\fin}^+$
		and it has no non-trivial compact normal subgroups.
	\end{itemize}

	The identification of $H$ with a subgroup of $\prod (\mathbf{H}_i)_{S_i^\infty}$ identifies $\Lambda$
	with a subgroup $\Sigma$ of the diagonal image of $\prod \mathbf{H}_i(K_i)$.
	We summarise some of the information we have so far in the following commutative diagram
	(arrows in this diagram are not named, but they are all explicit from our context):
%	\begin{equation} \label{commdiag8}
%	\xymatrix{
%		\Gamma|_U \ar[d]^{\simeq}\ar[r]^{\subset} & \Gamma \ar[d]^{\simeq}\ar[r]^{\subset} & H\times D\ar@{->>}[d]\ar@{->>}[r] &D\ar@{.>}[ddd]^{\text{?}}\\
%		\Delta \ar@{<~>}[d]^{\text{commen.}} \ar[r]^{\subset} & \Lambda \ar[d]^{\simeq}\ar[r]^{\subset} & H \ar[d]^{\simeq}\\
%		\Sigma\cap\prod \mathbf{H}_i(\calO_i) \ar[d]_{\text{f.i}}^{\subset}\ar[r]^{\subset} & \Sigma \ar[d]^{\subset}\ar[r] & \prod(\mathbf{H}_i)_{S_i^\infty}^0\\
%		\prod \mathbf{H}_i(\calO) \ar[r]^{\subset} & \prod\mathbf{H}_i(K_i) \ar[rr] & &\prod(\mathbf{H}_i)_{S_i^\fin}
%	}
%	\end{equation}
	\begin{equation} \label{commdiag8}
	\begin{tikzcd}
		\Gamma|_U \arrow[d,"\cong"]\arrow[r,hook] & \Gamma \arrow[d,"\cong"]\arrow[r,hook] & H\times D\arrow[d, twoheadrightarrow]\arrow[r, twoheadrightarrow] &D\arrow[ddd, dotted, "\text{?}"]\\
		\Delta \arrow[d, leftrightsquigarrow, ,"\text{commen.}"] \arrow[r,hook] & \Lambda \arrow[d,"\cong"]\arrow[r,hook] & H \arrow[d,"\cong"]\\
		\Sigma\cap\prod \mathbf{H}_i(\calO_i) \arrow[d,"\text{f.i}",hook]\arrow[r,hook] & \Sigma \arrow[d,hook]\arrow[r] & \prod(\mathbf{H}_i)_{S_i^\infty}^0\\
		\prod \mathbf{H}_i(\calO) \arrow[r,hook] & \prod\mathbf{H}_i(K_i) \arrow[rr] & &\prod(\mathbf{H}_i)_{S_i^\fin}
	\end{tikzcd}
	\end{equation}
	where the twiddle vertical arrow between $\Delta$ and $\Sigma\cap\prod \mathbf{H}_i(\calO_i)$ indicates that the images of these two groups in $\Sigma$ are commensurated. 
	The dotted arrow from $D$ to $\prod(\mathbf{H}_i)_{S_i^\fin}$ is not yet given -- its existence is the heart of the theorem. This is what we are now after.

	By the above $\Sigma$ contains a finite index subgroup of $\prod \mathbf{H}_i(\calO_i)$ and a finite index subgroup of $\Sigma$ 
	is contained in $\prod \mathbf{H}_i(\calO_i[S_i])$.
	Also for each $i$,
	the closure of the projection $\Sigma\to (\mathbf{H}_i)_{S_i^\fin}$ has no non-trivial compact normal subgroup.
	Denote by $Q<\prod (\mathbf{H}_i)_{S_i^\fin}$ the closure of the image of $\Sigma$.
	Denote by $C$ the commensurability class of $\Sigma\cap \prod \mathbf{H}_i(\calO_i)$ in $\Sigma$.
	By Lemma~\ref{lem:terminal}
	the embedding $\Sigma\to Q$ is a dense tdlc representation of the commensurability pair $(\Sigma,C)$
	and it is a terminal object in the category of dense tdlc representations of this pair
	(see \S\ref{sub:reps-of-comm} for definitions).
	Hence the path from $\Gamma$ to $Q$ in the commutative diagram below is a terminal dense tdlc representation of the pair $(\Gamma,[\Gamma|_U])$. 
	\[ 
		\begin{tikzcd}
		\Gamma|_U  \arrow[d, leftrightsquigarrow, "\text{commen.}"]\arrow[r] & \Gamma \arrow[d,"\simeq"]\arrow[r] & D\arrow[d,dotted,"\text{?}"]\\
		\Sigma\cap\prod \mathbf{H}_i(\calO_i) \arrow[r] & \Sigma \arrow[r] & Q \arrow[r,hook] & \prod(\mathbf{H}_i)_{S_i^\fin}
		\end{tikzcd}
	\]
	By the property of terminal objects we can uniquely realise the dotted arrow in the above diagram. By Lemma~\ref{lem:closedimage} the dotted arrow is a continuous epimorphism with compact kernel~$C$. 
	The composition $D\to Q \to \prod(\mathbf{H}_i)_{S_i^\fin}$
	is thus a continuous homomorphism with closed image and compact kernel $C$ which realises the dotted arrow in Diagram~(\ref{commdiag8}).

	Note that 
	$\Sigma$ is a subgroup of $\prod \mathbf{H}_i(K_i)$
	and $Q$ is the closure of the projection of $\Sigma$ to $\prod (\mathbf{H}_i)_{S_i^\fin}$
	such that $\Sigma$ contains a finite index subgroup of $\prod \mathbf{H}_i(\calO_i)$, a finite index subgroup of $\Sigma$ 
	is contained in $\prod \mathbf{H}_i(\calO_i[S_i])$,
	and for each $i$, $(\mathbf{H}_i)_{S_i^\fin}^+$ is contained in the projection of $Q$ to $(\mathbf{H}_i)_{S_i^\fin}$.
	By Lemma~\ref{lem:subthmcore}
	we conclude that $\Sigma$ in is an arithmetically constructed lattice.
	The theorem is now proved by considering the commutative diagram: 
	\[
	\begin{tikzcd}
	\Gamma \arrow[d,"\cong"]\arrow[r] & H\times D/C\ar[d,"\cong"]\\
	\Sigma \arrow[r] &  \prod(\mathbf{H}_i)_{S_i^\infty}^0\times Q  \arrow[r,hook] & \prod(\mathbf{H}_i)_{S_i}
	\end{tikzcd}
	\]
	This completes the proof of Theorem~\ref{T:arithmetic-core}.

\begin{appendix}

\section{A theorem of Borel revisited}

In this appendix we generalise Borel's Theorem 3(b) from \cite{Borel-density}.
Our proof is almost identical to Borel's original proof and we have no doubt the sophisticated 
reader is well aware of this generalisation.
Nevertheless, since the result had not been published elsewhere in the form that we need, 
we decided to add the short proof for the reader's convenience.

\begin{theorem} \label{thm:comm}
Let $K$ be a field of characteristic 0.
Let $\mathbf{G}$ be a connected affine $K$-algebraic group and 
$\mathbf Z$ its center.
Let $H<\mathbf{G}(K)$ be a Zariski dense subgroup.
Let $A$ be a $K$-algebra and let $C$ be the commensurator of $H$ in $\mathbf{G}(A)$.
Then the image of $C$ in $\mathbf{G}/\mathbf Z(A)$ is contained in $\mathbf{G}/\mathbf Z(A)$.
\end{theorem}

Before proving the theorem, let us recall some basic facts about the functor of points of a $K$-algebraic group.
For a $K$-algebraic group $\mathbf{G}$ and a $K$-algebra $A$ we set $\mathbf{G}(A)=\mathrm{Hom}_{K\mbox{\scriptsize -alg}}(K[\mathbf{G}],A)$,
where $K[\mathbf{G}]$ denotes the Hopf algebra of $K$-regular functions on $\mathbf{G}$.
Using the coproduct $\Delta: K[\mathbf{G}]\to K[\mathbf{G}]\otimes_K K[\mathbf{G}]$ we get a product structure on $\mathbf{G}(A)$,
by sending $x,y\in \mathbf{G}(A)$ to the map given by the following composition
\[
 K[\mathbf{G}] \overset{\Delta}{\xrightarrow{\hspace*{1cm}}} K[\mathbf{G}]\otimes_K K[\mathbf{G}] 
\overset{x\otimes y}{\xrightarrow{\hspace*{1cm}}} A\otimes_K A \overset{m}{\xrightarrow{\hspace*{1cm}}} A. 
\]
Further, $\mathbf{G}(A)$ has an identity, given by the augmentation $f\mapsto f(e)\in K<A$,
and the inverse of $x\in \mathbf{G}(A)$ is given by $x\circ S$, where $S:K[\mathbf{G}]\to K[\mathbf{G}]$ is the antipode, $Sf(g)=f(g^{-1})$.
Thus $\mathbf{G}(A)$ is a group.
The association $A\leadsto \mathbf{G}(A)$ is called {\em the functor of points} associated with $\mathbf{G}$,
see for example Waterhouse's book \cite{Waterhouse} for an account on algebraic group theory based on this point of view.

We set $A[\mathbf{G}]:=A\otimes_K K[\mathbf{G}]$.
Consider the conjugation action of $\mathbf{G}$ on itself.
This is a $K$-morphic action which gives rise to a $K$-morphic action of $\mathbf{G}/\mathbf Z$ on $\mathbf{G}$.
Consider the action map $a:\mathbf{G}/\mathbf Z\times \mathbf{G}\to\mathbf{G}$.
For every $g\in \mathbf{G}/\mathbf Z(A)$ the map
\[  K[\mathbf{G}] \overset{a^*}{\xrightarrow{\hspace*{1cm}}} K[\mathbf{G}/\mathbf Z]\otimes_K K[\mathbf{G}]
\overset{g\otimes\id}{\xrightarrow{\hspace*{1cm}}} A\otimes_K K[\mathbf{G}] = A[\mathbf{G}]
\]
extends uniquely to a $A$-algebra automorphism $A[\mathbf{G}]\to A[\mathbf{G}]$.
By this we define an action of $\mathbf{G}/\mathbf Z(A)$ on $A[\mathbf{G}]$ by $A$-algebra automorphisms.

\begin{lemma} \label{lem:comm}
Let $K$ be a field of characteristic 0.
Let $\mathbf{G}$ be an affine $K$-algebraic group and 
$\mathbf Z$ its center.
Let $A$ be a $K$-algebra.
Consider the conjugation representation of $\mathbf{G}/\mathbf Z(A)$ on $A[\mathbf{G}]$.
Considering $K[\mathbf{G}]$ as a subspace of $A[\mathbf{G}]$ and $\mathbf{G}/\mathbf Z(K)$ as a subgroup of $\mathbf{G}/\mathbf Z(A)$,
we have 
\[ \mathbf{G}/\mathbf Z(K)=\{g\in\mathbf{G}/\mathbf Z(A)~|~g(K[\mathbf{G}])\subset K[\mathbf{G}]\}.\]
\end{lemma}

\begin{proof}
Consider $K[\mathbf{G}/\mathbf Z]$ as a Hopf algebra and consider the comodule structure on $K[\mathbf{G}]$ associated with the conjugation representation.
By \cite[3.3]{Waterhouse}, $K[\mathbf{G}]$ is a direct union of finite dimensional subcomodules.
We let $V<K[\mathbf{G}]$ be a finite dimensional subcomodule which generates $K[\mathbf{G}]$ as a $K$-algebra
($K[\mathbf{G}]$ is a finitely generated $K$-algebra).
By the correspondence between subcomodules and subrepresentations, \cite[3.2]{Waterhouse}, 
$V$ is a representation of $\mathbf{G}/\mathbf Z$, in the sense that 
for every $K$-algebra $B$, 
$B\otimes_K V$ is a $\mathbf{G}/\mathbf Z(B)$-invariant subspace of $B[\mathbf{G}]$,
and we obtain a natural family of group homomorphisms $\mathbf{G}/\mathbf Z(B)\to \Aut_{B}(B\otimes_K V)$.
Considering the action of $\mathbf{G}/\mathbf Z(B)$ on $B[\mathbf{G}]$, it follows that the left hand side is  included in the right hand side in 
the following equation:
\begin{equation} \label{some-eq}
\{g\in\mathbf{G}/\mathbf Z(B)~|~g(K[\mathbf{G}])\subset K[\mathbf{G}]\} = \{g\in\mathbf{G}/\mathbf Z(B)~|~gV\subset V\},
\end{equation}
the other inclusion holds as the $K$-algebra generated in $B[\mathbf{G}]$ by $V$ is $K[\mathbf{G}]$.

We now take $B=\bar{K}$, an algebraic closure of $K$.
By definition of $\mathbf Z$ \cite[1.7]{borel-book}, $\mathbf{G}/\mathbf Z(\bar{K})$ acts faithfully on $\mathbf{G}(\bar{K})$.
It follows that its action on $\bar{K}[\mathbf{G}]$ is faithful as well,
and by the fact that $\bar{K}\otimes_K V$ generates $\bar{K}[\mathbf{G}]$ as a $\bar{K}$-algebra, we get that $\mathbf{G}/\mathbf Z(\bar{K})$ 
acts faithfully on $\bar{K}\otimes_K V$.
Thus the morphism $\mathbf{G}/\mathbf Z(\bar{K})\to \GL(\bar{K}\otimes_K V)$ is injective.
By \cite[Cor~1.4(a)]{borel-book} its image is a Zariski closed $K$-subgroup of $\GL(\bar{K}\otimes_K V)$,
and by \cite[Proposition~6.4]{borel-book} this morphism is a $K$-isomorphism onto its image since it is separable 
by the characteristic 0 assumption on $K$.

As $K$-vector spaces, $V\cong K^n$. Thus we may identify $\mathbf{G}/\mathbf Z$ with a $K$-algebraic subgroup $\mathbf{H}<\GL_n$.
As subgroups of $\GL_n(A)$ we have $\mathbf{H}(K)=\mathbf{H}(A)\cap\GL_n(K)$. 
%Indeed, letting $I$ be the defining ideal of $\mathbf{H}$ in $A=k[\GL_n]$,
%we identify $\mathbf{H}(A)$ and $\GL_n(k)$ with the sets of $I$-invariants and $K$-valued maps inside $\GL_n(A)=\mathrm{Hom}_K(A,K)$
%respectively,
%thus the intersection of the two is the set of $I$-invariant and $K$-valued maps,
%identified with $\mathbf{H}(K)=\mathrm{Hom}_K(A/I,K)$.
We conclude the first equality in
\[ \mathbf{H}(K)=\{g\in \mathbf{H}(A)~|~g(K^n)=K^n\}=\{g\in \mathbf{H}(A)~|~g(K^n)\subset K^n\}. \]
The second equality is a triviality since every $g$ is regular. 
We thus obtain
\[ \mathbf{G}/\mathbf Z(K)=\{g\in \mathbf{G}/\mathbf Z(A)~|~gV\subset V\}. \]
The proof is now complete by Equation~(\ref{some-eq}).
\end{proof}

\begin{proof}[Proof of Theorem~\ref{thm:comm}]
Let $g\in C$ be an element. We show that its image in $\mathbf{G}/\mathbf Z(A)$ -- denoted also by $g$ -- lies in $\mathbf{G}/\mathbf Z(K)$.

We consider the conjugation action of $\mathbf{G}(A)$ on $A[\mathbf{G}]$
and claim that $gK[\mathbf{G}]\subset K[\mathbf{G}]$.
By Lemma~\ref{lem:comm} this suffices. 
Let $x\in K[\mathbf{G}]$. 
We write 
\[ gx=x_0+\alpha_1x_1+\cdots +\alpha_nx_n, \quad \alpha_i\in A,~x_i\in K[\mathbf{G}], \]
where $1,\alpha_1,\ldots,\alpha_n$ are independent over $K$
and $x_i\neq 0$.
If $n=0$ then indeed $gx=x_0\in K[\mathbf{G}]$.

We assume $n\geq 1$ and argue by contradiction.
Note that $H\cap gHg^{-1}<H$ is of finite index, thus $H\cap gHg^{-1}$ is Zariski dense in $\mathbf{G}(K)$ since $\mathbf{G}$ is connected.
Thus
we can find $h\in H\cap gHg^{-1}<\mathbf{G}(K)=\mathrm{Hom}_{K\mbox{\scriptsize -alg}}(K[\mathbf{G}],K)$ such that $h(x_1)\neq 0$.
From 
\[ 
	h(x_0)+\alpha_1h(x_1)+\cdots +\alpha_nh(x_n)=h(gx)=(g^{-1}hg)(x) \in K 
\]
we get a non-trivial linear combination of $1,\alpha_1,\ldots,\alpha_n$ over $K$. This gives the desired contradiction. 
\end{proof}

\begin{remark} \label{rem:comm}
Neither Theorem~\ref{thm:comm} nor Lemma~\ref{lem:comm} are correct as stated if $K$ has positive characteristic.
Take for example $K$ to be the algebraic closure of $\mathbb F_2$, $\mathbf{G}=\SL_2$,
and $A=K[t]/(t^2)$.
Then $\mathbf{G}(K)=\SL_2(K)$ has trivial center $\mathbf Z$, thus $\mathbf{G}/\mathbf Z=\mathbf{G}$.
However, $\mathbf{G}/\mathbf Z(A)= \mathbf{G}(A)\cong \SL_2(K)\ltimes \mathfrak{sl}_2(K)$ has a non-trivial center
$L$ corresponding to the scalar matrices in $\mathfrak{sl}_2(K)$.
Then $L$ is not contained in $\mathbf{G}/\mathbf Z(K)$, but
\[ L<\{g\in\mathbf{G}/\mathbf Z(A)~|~g(K[\mathbf{G}])<K[\mathbf{G}]\}, \]
and for any $H$ in $\mathbf{G}(K)$, $L$ commensurates $H$ in $\mathbf{G}/\mathbf Z(A)$.

Correct statements are obtained by replacing the smooth subgroup $\mathbf Z$ with a scheme theoretically defined center $\mathcal{Z}$. 
For example, for $\mathbf{G}=\SL_2$, regardless the field of definition, $\mathbf{G}/\mathcal{Z}\cong \PGL_2$.
In characteristic 0 we always have $\mathcal{Z}=\mathbf Z$.
The proof of the corrected statement of the theorem follows formally from the corrected lemma.
We will not give the proof of the corrected lemma 
but merely remark that it is verbatim the proof above, 
upon replacing the notion of ``faithfulness of an action" by a scheme theoretical analogue.
Moreover, there is no need to pass to an algebraic closure anymore.

Working with the scheme theoretical center has advantages even if one is concerned only with the classical setting.
For example, if $K$ is reduced then $\mathbf{G}/\mathcal{Z}(K)=\mathbf{G}/\mathbf Z(K)$.
To see this one observes that the algebra representing $\mathbf{G}/\mathcal{Z}$, modulo its nil-radical,
equals the algebra representing $\mathbf{G}/\mathbf Z$.
\end{remark}

The following application of Theorem~\ref{thm:comm} is needed in Section~\ref{sec: Margulis commensurator rigidity}. 

\begin{theorem} \label{app:comm}
We retain the setup in Subsection~\ref{sub: intro-notation}. 
Let $\mathbf{H}$ be a connected adjoint semi-simple $K$-algebraic group. 
Assume in 
addition that $\mathbf{H}(\calO[S])$ is Zariski dense in $\mathbf{H}$. 
Then the commensurator of $\mathbf{H}(\calO[S])$ in $\mathbf{H}_S$ (via the diagonal embedding) is exactly $\mathbf{H}(K)$.
\end{theorem}

\begin{remark}
The subgroup $\mathbf{H}(\calO[S])<\mathbf{H}(K)$ depends on a choice of 
embedding $\mathbf{H}<\GL_n$ but not its commensurability class. Since $\mathbf{H}$ is connected its Zariski-density is unambiguous.
\end{remark}

\begin{proof}
Consider the locally compact totally disconnected group given by the restricted product
\[
	{\prod}^{'}_{\nu\in \calV-\calV^\infty-S} \mathbf{H}(K_\nu),
\] 
which is taken with respect to the compact 
open subgroups $\mathbf{H}(\calO_\nu)<\mathbf{H}(K_\nu)$.
Then 
\[
	{\prod}_{\nu\in \calV-\calV^\infty-S} \mathbf{H}(\calO_\nu)
\] 
is a compact open subgroup, hence it is commensurated.
Its preimage under the diagonal embedding 
\[
	\mathbf{H}(K)\to {\prod}^{'}_{\nu\in \calV-\calV^\infty-S}  \mathbf{H}(K_\nu)
\]
is $\mathbf{H}(\calO[S])$.
By Lemma~\ref{L:a-normal}, $\mathbf{H}(\calO[S])$ is commensurated by $\mathbf{H}(K)$.

We need to show that the commensurator group of $\mathbf{H}(\calO[S])$ 
inside $\mathbf{H}_S$ is contained in the diagonal image of $\mathbf{H}(K)$.
Let $A$ be the $K$-algebra of $S$-adeles, that is, the restricted product $\prod'_{\nu\in S} K_\nu$ 
taken with respect to the subgroups $\calO_\nu<K_\nu$.
In \cite[pp.~ 249-250]{Platonov+Rapinchuk} the identification $\mathbf{H}(A)\cong \mathbf{H}_S$ is explained.
Under this identification, $\mathbf{H}(K)$ is identified with its diagonal image.
Now apply Theorem~\ref{thm:comm} and use the assumption that $\mathbf{H}$ is adjoint, hence has a trivial center.
\end{proof}

For the special case where $S$ is a set of archimedean places we get the following.

\begin{theorem} \label{thm:appcomm}
We retain the setup in Subsection~\ref{sub: intro-notation}. 
Let $S\subset \calV^{\infty}$ be a set of archimedean places. 
Let $\mathbf{H}$ be a connected adjoint semi-simple $K$-algebraic group. 
Assume that there exists $\nu\in \calV^\infty$ such that the Lie group $\mathbf{H}(K_\nu)$ is not compact.
Then the commensurator of $\mathbf{H}(\calO)$ in $\prod_{\nu\in S}\mathbf{H}(K_\nu)$ is exactly $\mathbf{H}(K)$.
\end{theorem}

\begin{proof}
This follows from Theorem~\ref{app:comm}, once we show that $\mathbf{H}(\calO)$ is Zariski dense in $\mathbf{H}$. 
By \cite[Theorem~3(a)]{Borel-density} $\mathbf{H}(\calO)$ is Zariski dense if and only if $R_{K/\bbQ}\mathbf{H}(\bbR)$ is non-compact 
where $R_{K/\bbQ}$ denotes the restriction of scalars from $K$ to $\bbQ$. 
By applying~\cite[Theorem~I.3.1.4(i)]{Margulis:book} to the real place,
\[ 
	R_{K/\bbQ}\mathbf{H}(\bbR) \cong \prod_{\nu\in V^\infty}  \mathbf{G}(K_\nu). 
\]
The right hand side is non-compact by assumption. 
\end{proof}

In view of Remark~\ref{rem:comm}, Theorem~\ref{app:comm}
has the following analogue in positive characteristic.
The proof is the same as of Theorem~\ref{app:comm},
taking into account that for adjoint groups the scheme theoretical center is trivial.

\begin{theorem}
Let $K$ be field of rational functions over a finite field, $\calO$ its ring of polynomials 
and $\calV$ the set of places of $k$.
We fix a subset $S\subset \calV$.
Let $\mathbf{H}$ be a connected adjoint semi-simple $K$-algebraic group.
We assume that $\mathbf{H}(\calO[S])$ is Zariski dense in $\mathbf{H}$.
Then the commensurator of $\mathbf{H}(\calO[S])$ in $\prod'_{\nu\in S}\mathbf{H}(K_\nu)$ is exactly $\mathbf{H}(K)$.
\end{theorem}

\end{appendix}

\end{document}